\documentclass[11pt]{amsart}

\usepackage{amscd}
\usepackage{amsthm}
\usepackage[centertags]{amsmath}
\usepackage{amsfonts}
\usepackage{newlfont}
\usepackage{graphicx}
\usepackage[usenames]{color}
\usepackage{amsfonts, amssymb}
\usepackage{mathrsfs}
\usepackage{latexsym}
\usepackage{tikz}
\usepackage{tikz}\usepackage{verbatim}
\usepackage{tabulary,tabularx}
\usepackage[all]{xy}
\usepackage[latin1]{inputenc}

\newtheorem{theorem}{Theorem}[section]

\newtheorem{proposition}[theorem]{Proposition}

\newtheorem{lemma}[theorem]{Lemma}

\theoremstyle{definition}

\theoremstyle{definition}

\theoremstyle{remark}

\newcommand{\PP}{\mathcal P}
\newcommand{\LL}{\mathcal L}

\newcommand{\zR}{\mathbb R}
\newcommand{\zN}{\mathbb N}

\renewcommand{\H}{H(\mathbb{C})}
\newcommand{\C}{\mathbb C}
\newcommand{\zZ}{\mathbb Z}
\newcommand{\sub}{\subseteq}
\newcommand{\zT}{\mathbb T}

\newcommand{\f}{\frac}

\renewcommand{\l}{\left(}
\renewcommand{\r}{\right)}

\title{Hypercyclic homogeneous polynomials on $\H$.}
\author{Rodrigo Cardeccia, Santiago Muro}

\thanks{Partially supported by ANPCyT PICT 2015-2224, UBACyT 20020130200052BA and CONICET}

\address{DEPARTAMENTO DE MATEM\'ATICA - PAB I,
FACULTAD DE CS. EXACTAS Y NATURALES, UNIVERSIDAD DE BUENOS AIRES, (1428) BUENOS AIRES, ARGENTINA AND CONICET} \email{rcardeccia@dm.uba.ar} \email{smuro@dm.uba.ar}

\begin{document}
 \begin{abstract}
It is known that homogeneous polynomials on Banach spaces cannot be hypercyclic, but there are examples of hypercyclic homogeneous polynomials on some non-normable Fr\'echet spaces. We show the existence of  hypercyclic polynomials on $\H$, by exhibiting a concrete polynomial which is also the first example of a frequently hypercyclic homogeneous polynomial on any $F$-space.  
 We prove that the homogeneous polynomial on $\H$ defined as the product of a translation operator and the evaluation at 0  is mixing, frequently hypercyclic and chaotic. We prove, in contrast, that 
some natural related polynomials fail to be hypercyclic.
\end{abstract}

\subjclass[2010]{
47H60  %Multilinear and polynomial operators
37F10, %Complex dynamical systems - Polynomials; rational maps; entire and meromorphic functions
47A16, %Cyclic vectors, hypercyclic and chaotic operators
30D20, %Functions of a complex variable, Entire functions, general theory 
30K99  % Functions of a complex variable -Universal holomorphic functions -None of the above, but in this section
}
 \maketitle
 \section{Introduction}
 Let $X$ be an $F$-space. A function $T:X\rightarrow X$ is said to be \emph{hypercyclic} if there exists $x\in X$ such that its orbit, $Orb_T(x):=\{T^n(x):n\in\zN\},$ is dense in $X$. In this case, $x$ is called a \emph{hypercyclic vector}. The space $\H$ of entire functions, with the compact open topology, was of crucial importance since the beginnings of the theory of hypercyclic linear operators. Indeed, the first example of a hypercyclic operator was found by Birkhoff in \cite{Bir29}. There, he showed that there exists an entire function $g\in H(\C)$ whose translations by natural numbers approximate uniformly on compact sets any other entire function, i.e. the translation operator $\tau_1f(z)=f(z+1)$ acting on the space of entire functions $\H$ is hypercyclic. Later, MacLane \cite{Mac52} exhibited  the second example of a hypercyclic operator, also on $\H$, proving that the differentiation operator $Df(z)= f'(z)$  is also hypercyclic. 
 
 At the beginning of the 1990 decade, the theory of hypercyclic operators began to have a great development. An article that inspired much of the subsequent work was the  seminal paper of Godefroy and Shapiro \cite{GodSha91}, where the authors proved (among other things) an important generalization of the results of Birkhoff and MacLane. More recently, the concept of frequently hypercyclic operator was introduced in \cite{BayGri04}, and shortly after, the operators considered by Birkhoff, MacLane, Godefroy and Shapiro were shown to be also frequently hypercyclic \cite{BayGri06,BonGro06}. {For a systematic treatment of hypercyclic operators and related topics see the recent books \cite{GroPer11,BayMat09} and the references therein.}

 As a natural extension of the linear theory, one may study orbits of (non-linear)  polynomial operators on $F$-spaces.
 The first results  were obtained by Bernardes in \cite{Ber98}, in the context of homogeneous polynomials acting on Banach spaces. Maybe surprisingly, he showed that no homogeneous
 polynomial, of degree $\ge 2$, acting on a Banach space can be hypercyclic. In contrast, if the $F$-space is not normable, it may support hypercyclic homogeneous polynomials. The first to realize this fact was Peris \cite{Per99,Per01}. As it is natural, the space where he sought a homogeneous hypercyclic polynomial was $\H$. Unfortunately, the example he gave was not well defined. However, he was able to construct another example, this time on the space $\C^\zN$, the Fr\'echet space of all complex sequences. He showed  that the polynomial $(a_n)\mapsto (a_{n+1}^2)$ is not only hypercyclic but also chaotic on $\C^\zN$.
   
 After the example of Peris, some other hypercyclic homogeneous polynomials were presented, on some K\"othe echelon spaces (including the space $H(\mathbb D)$, see \cite{MarPer10}) and on some spaces of differentiable functions on the real line \cite{AroMir08}. But there are, up to our knowledge, no examples of hypercyclic homogeneous polynomials on $\H$. There are also no examples of frequently hypercyclic homogeneous polynomials on any $F$-space. Given the key role of $\H$ in the theory of linear dynamics, we believe it is desirable
 to exhibit examples of hypercyclic homogeneous polynomials on $\H$.
  
 There are also some other articles investigating the dynamics of non-homogeneous polynomials (\cite{bernal2005backward,BerPer13,jung2017mixing,kim2012numerically,MarPer09,MarPer10,peris2003chaotic}) and of	 multilinear mappings (\cite{BesCon14,GroKim13}) on infinite dimensional spaces. For example, in the recent paper \cite{BerPer13}, the existence of hypercyclic (non-homogeneous) polynomials of arbitrary positive degree is shown on any infinite dimensional Fr\'echet space.
 
 In this note we show that the 2-homogeneous polynomial $P(f)=f(0)\cdot\tau_1f$ defined on $\H$ is  mixing, chaotic and frequently hypercyclic. 
 In contrast, we prove that the polynomial $P(f)=f(0)\cdot f'$ is not hypercyclic on $\H$.

\section{A hypercyclic polynomial on $\H$}
In this section we prove Theorem \ref{main}, our main result, which states that there is a very natural hypercyclic homogeneous polynomial on $\H$. Let us first recall some definitions.
If $T$ is a mapping acting on a topological space $X$, $T$ is said to be \emph{transitive} if for each nonempty open sets $U,V\subset X$ there exists $n\in \zN$ such that
$T^n(U)\cap V\ne\emptyset$. If there exists $n_0$ such that $T^n(U)\cap V\ne\emptyset$ for every $n\geq n_0$, the mapping is said to be \emph{mixing}.
Clearly a mixing map is transitive and by Birkhoff's Transitivity Theorem, if the map is continuous and the underlying space is a complete separable metric space without isolated points, then
the map is transitive if and only if it is hypercyclic.

A set $A\sub \zN$ is said to have positive \emph{lower density} if
$$\liminf_n \f{\#\{x\in A: 0\leq x\leq n\}}{n}>0,$$
where $\#$ denotes the cardinality of the set. We say that a map $T$ is \emph{frequently hypercyclic} if there exists some $x\in X$ such that for every nonempty open set $U$, the set
$\{n\in\mathbb N:T^n(x)\in U\}$ has positive lower density.
Finally $T$ is said to be \emph{chaotic} if it is hypercyclic and has a dense set of periodic vectors. 

Let $X$ be an $F$-space. A mapping $P:X\to X$ is said to be a $d$-homogeneous polynomial, $P\in\mathcal P(^dX)$, if
$P$ is the restriction to the diagonal of some $d$-multilinear map $L\in \LL(^d X;X)$, that is, 
$$P(x)=L(\overbrace{x,...,x}^d).$$ 

We will be dealing with homogeneous polynomials acting on the space $\H$ of entire functions, which endowed with the compact open topology is a Fr\'echet space. The seminorms
$$\|f\|_K:=\sup_{z\in K} |f(z)|,$$
where $K$ is a  compact set, define the topology in $\H$. Thus, the sets 
$$U_{\epsilon,f,R}=\{h\in\H: \|h-f\|_{B(0,R)}<\epsilon\},$$
with $\epsilon,R>0$ form a basis of open neighborhoods of $f\in\H$.

\begin{theorem}\label{main}
The polynomial $P \in \PP(^2\H)$ defined by $$P(f)(z)=f(0)\cdot f(z+1)$$ is mixing, chaotic and frequently hypercyclic.
\end{theorem} 
 Observe that $P^n(f)(z)=c_n(f) f(z+n)$ where 
 \begin{equation}\label{c_n}
 c_n(f)=f(0)^{2^{n-1}}\cdot f(1)^{2^{n-2}}\cdot\ldots \cdot f(n-1).
 \end{equation}

 \subsection*{Proof that $P$ is mixing}
 Let $U$ and $V$ be open sets. We can suppose that 
 $U=U_{\epsilon,f,R}$ and $V=U_{\epsilon,g,R}.$
 Also we may suppose that $R\notin \zN$ and that  $f$ and $g$ do not have zeros in $\zZ$. 
 
 By Runge's Theorem we can find, for $n$ large enough, a polynomial $p$  such that $p\in U$ and $p(\cdot+n)\in V$. We assert that a more careful application of Runge's Theorem allows us to obtain a polynomial $p$ that also satisfies $c_n(p)\sim 1$.
 
Let $n_0\in\zN$ such that $n_0>2R+2$, and fix $n\ge n_0$. This implies that $B(0,R)\cap B(n,R)=\emptyset$ and that 
we can define open balls $B^1$, $B^2\sub \C$ such that $\{B(0,R),B^1,B^2,B(n,R)\}$ are pairwise disjoint and such that $\lfloor R\rfloor +1\in B^1$, and $\lfloor R\rfloor +2,\ldots, n-\lfloor R \rfloor-1 \in B^2$, where $\lfloor R \rfloor$ denotes the integer part of $R$.
See \textbf{Fig. \ref{fig1}}.
\begin{figure}
\centering
 \begin{tikzpicture}
\draw  (2,0) circle [radius=1];
\draw  (4,0) circle [radius=0.5];
% \draw  (8.25,0) circle [radius=3.5];
\draw  (13,0) circle [radius=1];
\draw  (8.25+0.25*7*1.4142,-0.25*7*1.4142) arc [radius=3.5, start angle=-45, end angle= 45];
\draw  (8.25-0.25*7*1.4142,0.25*7*1.4142) arc [radius=3.5, start angle=135, end angle= 225];
\draw  (0,0) --(6.5,0);
\draw [dashed] (6.5,0) -- (10,0);
\draw  (10,0) --(15,0);
\node [below,scale=0.7] at (2,0) {0};
\node [below,scale=0.7] at (2,-1) {$B(0,R)$};
\node [below,scale=0.7] at (3.1,0) {$R$};
\node [below,scale=0.7] at (4,0){$\lfloor R \rfloor +1$};
\node [below,scale=0.7] at (4,-0.5) {$B^1$};
\node [below,scale=0.7] at (5.2,0){$\lfloor R \rfloor +2$};
\node [below,scale=0.7] at (13,0) {$n$};
\node [below,scale=0.7] at (13,-1) {$B(n,R)$};
\node [below,scale=0.7] at (12.2,0) {$n-R$};
\node [below,scale=0.7] at (11,0) {$n-\lfloor R \rfloor-1$};
\node [below,scale=1] at (8.25,-2) {$B^2$};
\draw [fill] (2,0) circle [radius=0.025];
\draw [fill] (3,0) circle [radius=0.025];
\draw [fill] (13,0) circle[radius=0.025];
\draw [fill] (12,0) circle[radius=0.025];
\draw [fill] (11,0) circle[radius=0.025];
\draw [fill] (10,0) circle[radius=0.025];
\draw [fill] (4,0) circle [radius=0.025];
\draw [fill] (5,0) circle [radius=0.025];
\draw [fill] (6,0) circle [radius=0.025];
\end{tikzpicture}
\caption{The open sets $B(0,R)$, $B^1$, $B^2$ and $B(n,R)$.}
\label{fig1}
\end{figure}
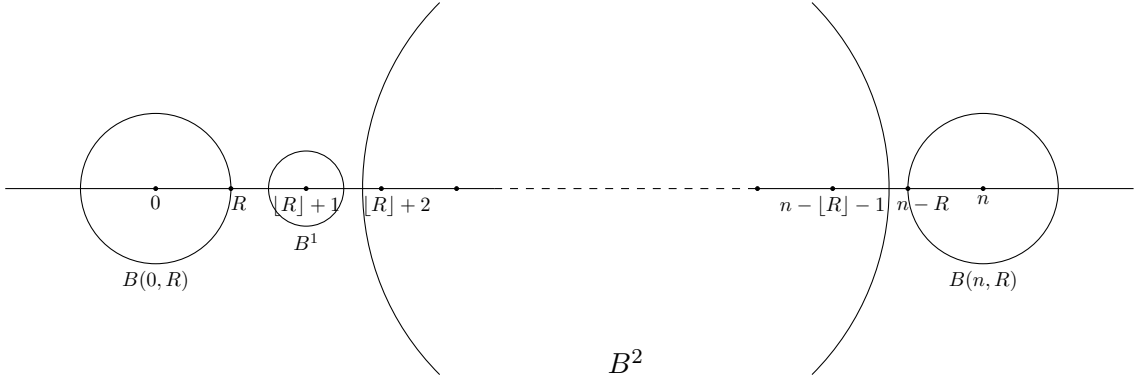

Define $\tilde g(z)= g(z-n)$ and $\alpha$ any $2^{n-\lfloor R \rfloor-2}$th-root of the number
$$f(0)^{2^{n-1}}\cdot\ldots \cdot f(\lfloor R \rfloor)^{2^{n-\lfloor R \rfloor-1}}\cdot 1^{2^{n-\lfloor R \rfloor-3}}\cdot\ldots \cdot1^{2^{\lfloor R \rfloor}}\cdot \tilde g(n-\lfloor R\rfloor)^{2^{\lfloor R \rfloor-1}}\cdot\ldots\cdot\tilde g(n-1).$$

Also consider the perturbed open sets in $\H$, 
\begin{align*}
 U_k&=\left\{h\in \H: \|f-h\|_{B(0,R)}<\f{\epsilon}{k}\right\},\\
 \tilde V_k&=\left\{h\in\H:\|\tilde g-h\|_{B(n,R)}<\f{\epsilon}{k}\right\},\\
 W^1_k&=\left\{h\in\H:\sup_{z\in B^1}\left|h(z)-\f{1}{\alpha}\right|<\f{1}{k}\right\} \text{ and }\\
W^2_k&=\left\{h\in\H:\sup_{z\in B^2}|h(z)-1|<\f{1}{k}\right\}.
 \end{align*}
 By Runge's Theorem we can find,  for each $k$, a polynomial $p_k$ in $U_k\cap W^1_k\cap W^2_k\cap \tilde V_k$.

 Observe that for $j\in \mathbb N$, we have 
\begin{align*}
 p_k(j)\rightarrow \begin{cases}
                    f(j)&\text{ if } j\leq \lfloor R\rfloor;\\
                    \f{1}{\alpha}& \text{ if } j=\lfloor R \rfloor +1;\\
                    1                   & \text{ if } \lfloor R \rfloor +1<j\leq n-\lfloor R\rfloor-1;\\
                    \tilde g(j)&\text{ if } n-\lfloor R\rfloor-1<j\leq n-1,
                   \end{cases}
\end{align*}
as $k\to \infty$.
Also, by definition of $\alpha$, $c_n(p_k)\rightarrow 1$ as $k\to \infty$. Thus, for large $k$ we have
$$\|c_n(p_k)p_k-\tilde g\|_{B(n,R)}\leq \f{\epsilon}{2} + |c_n(p_k)-1|\|p_k\|_{B(n,R)}\leq \f{\epsilon}{2} + |c_n(p_k)-1|\l \f{\epsilon}{2}  + \|\tilde g\|_{B(n,R)}\r< \epsilon.$$

Therefore, we can find a polynomial $p_k$ with
$$ 
 \|f-p_k\|_{B(0,R)}<\epsilon \quad\text{ and }\quad \|g-P^n(p_k)\|_{B(0,R)}=\|\tilde g-c_n(p_k)p_k\|_{B(n,R)}<\epsilon.$$
This proves that $P$ is mixing.

\subsection*{Proof that $P$ is chaotic}
Observe that a periodic vector for $P$ is a quasiperiodic function, that is, there exist $\alpha\in\C$ and $n\in\zN$ such that $f(z+n)=\alpha f(z)$. If this happens, then the homogeneity of $P$ forces
\begin{equation}\label{periodic2}
\l\f{1}{c_n(f)\alpha}\r^\f{1}{2^n-1}f
\end{equation}
 to be an $n$-periodic vector for $P$. Note also that if $f$ is a periodic vector for $P$, then $\lambda f$ is not necessarily a periodic vector for $P$.

 It is known that the set of periodic functions is dense in $\H$. To prove that $P$ is chaotic we will show that the set of periodic functions satisfying that $\l\f{1}{c_n(f)}\r^\f{1}{2^n-1}\sim 1$ is also dense in $\H$. So, it will be useful to have a good characterization of
the periodic functions. Define an infinite segment $L$, beginning at zero, so that $L\cap\zT$ is not a root of the unity and $\theta=\f{arg(L)}{2\pi}\in (\f 14,\f{3}{4})$, and define $\Omega=\C-L$ (so that
a branch of the logarithm may be defined on $\Omega$). Then, since $e^{\f{2\pi i}{n}z}$ maps any band  $(n(\theta+k),n(1+\theta+k))\times i\zR$ to $\Omega$,  we have the following.
\begin{lemma}\label{periodic}
 If an entire function $f$ is $n$-periodic then there exist $g$ holomorphic on $\Omega$ such that
 $$f(z)=g(e^{\f{2\pi i}{n}z})$$
 for all $z$ belonging to any band of the form $(n(\theta+k),n(1+\theta+k))\times i\zR$. Reciprocally, if $f(z)=g\l e^{\f{2\pi i}{n}z}\r$ for all $z\in\C$, then
 $f$ is $n$-periodic.
\end{lemma}

We now begin with our proof of the chaoticity of $P$.

Let $U=\{h\in\H:\|h-g\|_{B(0,R)}<\epsilon\}$ be a  nonempty open set of $\H$ with  $R\not \in \zN.$ Our goal is to find, for some $n\in\mathbb N$, an $n$-periodic function $f\in U$ so that also
$c_n(f)^\f{-1}{2^n-1}f\in U$. By \eqref{periodic2}, this implies that $c_n(f)^\f{-1}{2^n-1}f$ is a periodic vector for $P$ and therefore, the set of periodic vectors is dense in $\H$.

Take $n_0\in\mathbb N$ so that $n_0>4R$. 
Since the periodic functions with period greater than $n_0$ are dense in $\H$ \cite[Sublemma 7]{AroMar04} , there exists, for some $n>n_0$, an $n$-periodic function $f$ with $\|f-g\|_{B(0,R)}<\f{\epsilon}{2}$. We may also suppose that $f(j)\ne 0$ for every $j\in \mathbb Z$.

Now take $k\in \zZ$ such that $B(0,R)$ is contained in the band $(\theta+k,n+\theta+k)\times i\zR$. Thus, by the previous Lemma, $f(z)=h(e^\f{2\pi iz}{n})$ for every $z\in B(0,R)$ for an appropriate holomorphic function $h$ on $\Omega$.  
 Instead of applying Runge's Theorem to
the  function $f$ we will apply it to $h$. The function $e^\f{2\pi iz}{n}$ maps $\zN_0$ to $G_n$, the $n$-th roots of the unity, which we will denote $\omega_0,\ldots,\omega_{n-1}$. Thus, $h(\omega_j)\neq 0$ for every $\omega_j\in G_n$ and
$$P^n(h\circ e^\f{2\pi iz}{n})=\tilde c_n(h)\cdot h\circ e^\f{2\pi iz}{n},$$
where $\tilde c_n(h):= h(\omega_0)^{2^{n-1}}\ldots h(\omega_{n-1})=c_n(h\circ e^\f{2\pi iz}{n})$. Consider  $B_1=\{e^{\f{2\pi i z}{n}}:\ z\in B(0,R)\}$ and observe that 
$\omega_0,\omega_1,\ldots,\omega_{\lfloor R \rfloor}$, and $\omega_{n-\lfloor R \rfloor},\ldots\omega_{n-1}$ are all in $B$ while
$\omega_{\lfloor R \rfloor+1},\ldots,\omega_{n-\lfloor R \rfloor-1}$ are in $(\overline {B_1}^c)$. Also, since $n>4R$ and $\theta\in (\f 14,\f34)$, $B_1\sub \Omega$ and
$h$ is holomorphic on $B_1$. Runge's Theorem allows us to find $\tilde h$ such that $\tilde h$ is close to $h$ on $B_1$
and at the same time $\tilde c_n (\tilde h)$ is close to  $1$. Indeed, choose $B_2$ and $B_3$ open sets so that
$\omega_{\lfloor R \rfloor+1}\in B_2$, $\omega_{\lfloor R \rfloor+2},\ldots,\omega_{n-\lfloor R \rfloor-1}$ are in $B_3$ and $B_1,B_2,B_3$ are pairwise disjoint.
See \textbf{Fig. \ref{fig2}}.

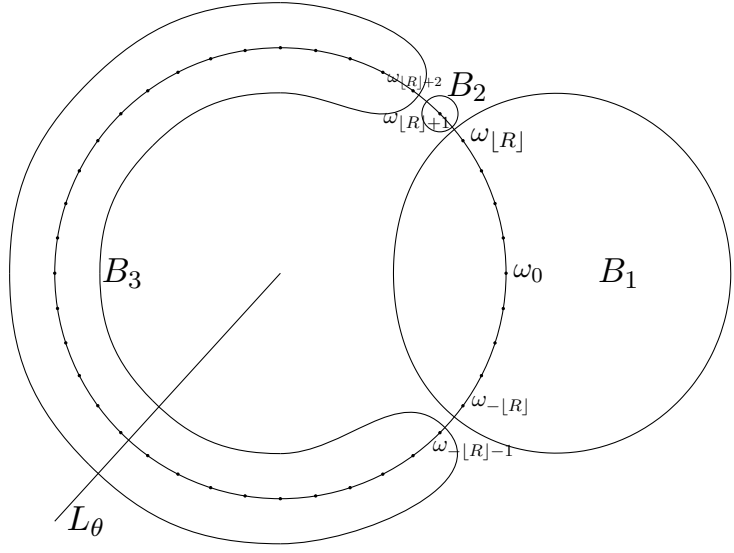
\begin{figure}
 \centering
\begin{tikzpicture}[scale=3]
%%%%%%%%%%%%%%%%%%%%%%%%%%%%%%%%%%%%n=100,R=20%%%%%%%%%%%%%%%%%%%%%%%%%%%%%%%%%%%%%%%%%%%%
%n=40,R=10.5
\draw  (0,0) circle [radius=1];

%RAICES DE LA UNIDAD
\foreach \x in {0,1,2,...,39}
  \draw  [fill] ({cos (0.05*pi*\x r)},{sin (0.05*pi*\x r)}) circle [radius=0.005];

%BOLA ALREDEDOR DE LA PRIMER RAIZ  
  \draw  ({cos (0.05*pi*5 r)},{sin (0.05*pi*5 r)}) circle [radius=0.08];

%CHORIZO 2%%%%%%%%%%%%%%%%%%%%%%%%%%%%%%%%%%%%%%%%
  \coordinate (1)  at ({cos (0.05*pi*5.8 r)},{sin (0.05*pi*5.8 r)}) {};
   \coordinate (2)  at ({1.2*cos (0.05*pi*15 r)},{1.2*sin (0.05*pi*15 r)}) {};
   \coordinate (3)  at ({1.2*cos (0.05*pi*20 r)},{1.2*sin (0.05*pi*20 r)}) {}; 
  \coordinate (4)  at ({1.2*cos (0.05*pi*25 r)},{1.2*sin (0.05*pi*25 r)}) {};
  \coordinate (5)  at ({1.2*cos (0.05*pi*30 r)},{1.2*sin (0.05*pi*30 r)}) {};
  \coordinate (6)  at ({cos (0.05*pi*-4.8 r)},{sin (0.05*pi*-4.8 r)}) {};
  \coordinate (7)  at ({0.8*cos (0.05*pi*30 r)},{0.8*sin (0.05*pi*30 r)}) {};
  \coordinate (8)  at ({0.8*cos (0.05*pi*25 r)},{0.8*sin (0.05*pi*25 r)}) {};  
  \coordinate (9)  at ({0.8*cos (0.05*pi*20 r)},{0.8*sin (0.05*pi*20 r)}) {};
\coordinate (10)  at ({0.8*cos (0.05*pi*15 r)},{0.8*sin (0.05*pi*15 r)}) {};
\coordinate (11)  at ({0.8*cos (0.05*pi*10 r)},{0.8*sin (0.05*pi*10 r)}) {};
\coordinate (12)  at ({1.2*cos (0.05*pi*10 r)},{1.2*sin (0.05*pi*10 r)}) {};

\draw (1) [-] to  [out=54,in=0] (12);
\draw (12) [-] to  [out=180,in=45] (2);
\draw (2) [-] to  [out=225,in=90] (3);
\draw (3) [-] to  [out=270,in=135] (4);
\draw (4) [-] to  [out=315,in=180] (5);
\draw (5) [-] to  [out=0,in=315] (6);
\draw (6) [-] to  [out=135,in=0] (7);
\draw (7) [-] to  [out=180,in=315] (8);
\draw (8) [-] to  [out=135,in=270] (9);
\draw (9) [-] to  [out=90,in=225] (10);
\draw (10) [-] to  [out=45,in=180] (11);
\draw (11) [-] to  [out=0,in=234] (1);

%%%%%%%%%%%%%%%%%%%%%%%%%%%%%%%%%%%%%%%%%%%%%%%%%%%%%%%
%INTENTO DE CHORIZO 1, N=40, R=4.5
\draw[smooth, domain=0:2*pi,samples=100] plot ({exp(-0.11*2*pi*sin(\x r))*cos(0.11*2*pi*cos(\x r) r)},{exp(-0.11*2*pi*sin(\x r))*(sin(0.11*2*pi*cos(\x r) r))});

%%%%%%%%%%%%%%%%%%%%%%%%%%%%%%%%%%%%%%%%%%%%%%%%%%%%%%
%RECTA  L_theta
 
\draw (0,0) -- (-1,-1.1); 
 
%INDICE%%%%%%%%%%%%%%%%
%%%%%%%%%%%RAICES DE LA UNIDAD
\node at ({cos (0.0 r)+0.1},{sin (0.0 r)}) {$\omega_0$};
\node  at ({cos (0.05*pi*4 r)+0.16},{sin (0.05*pi*4 r)}) {$\omega_{\lfloor R\rfloor}$};
\node [scale=0.8] at ({cos (0.05*pi*5 r)-0.1},{sin (0.05*pi*5 r)-0.04}) {$\omega_{\lfloor R\rfloor+1}$};
\node [scale=0.65] at ({cos (0.05*pi*6 r)+0.01},{sin (0.05*pi*6 r)+0.04}) {$\omega_{\lfloor R\rfloor+2}$};
\node [scale=0.8] at ({cos (0.05*pi*-4 r)+0.17},{sin (0.05*pi*-4 r)}) {$\omega_{-\lfloor R\rfloor}$};
\node [scale=0.8] at ({cos (0.05*pi*-5 r)+0.15},{sin (0.05*pi*-5 r)-0.08}) {$\omega_{-\lfloor R\rfloor-1}$};

%%%%%%%%%%%%%%%%%%%%%%%%
%RECTA L_\theta
\node [right,scale=1.2] at (-1,-1.1) {$L_\theta$};
%%%%%%%%%%%%%%%%%%%%%%%
%CONJUNTOS
\node [scale=1.2] at (1.5,0) {$B_1$};
\node [scale=1.2] at ({cos (0.05*pi*5 r)+0.12},{sin (0.05*pi*5 r)+0.12}) {$B_2$};
\node [scale=1.2] at (-0.7,0) {$B_3$};

\end{tikzpicture}
\caption{
 Shape and location of the sets $B_1$, $B_2$, $B_3$, $L_\theta$ and $G_n$.}
\label{fig2}
\end{figure}
Now define $U_1^l,U_2^l,U_3^l$ open sets in $\H$ as
\begin{align*}
  U_1^l&= \left\{g\in\H: \|g-h\|_{B_1}<\f{\epsilon}{l}\right\},\\
  U_2^l&= \left\{g\in\H: \sup_{z\in B_2}\left|g(z)-\frac1{\alpha}\right|<\f{\epsilon}{l}\right\},\\
  U_3^l&= \left\{g\in\H: \sup_{z\in B_3}\left|g(z)-1\right|<\f{\epsilon}{l} \right\};
\end{align*}
where $\alpha$ is any $2^{n-\lfloor R\rfloor-2}$ th-root of the number
$$h(\omega_0)^{2^{n-1}}\cdot\ldots\cdot h(\omega_{\lfloor R \rfloor})^{2^{n-\lfloor R\rfloor-1}}\cdot 1^{2^{n-\lfloor R\rfloor-3}}\cdot\ldots \cdot1^{2^{\lfloor R\rfloor}}\cdot h(\omega_{n-\lfloor R \rfloor})^{2^{\lfloor R\rfloor-1}}\cdot\ldots \cdot h(\omega_n).$$ 

 By Runge's Theorem we can find, for every $l$, a polynomial $h_l\in U_1^l\cap U_2^l\cap U_3^l$. By the choice of $h_l$ and $\alpha$, $\tilde c_n(h^l)\rightarrow 1$ and $\|h_l-h\|_{B_1}\rightarrow 0$ as $l$ tends to infinity.
Thus, $$\|c_n(h_l\circ e^{\f{2\pi iz}{n}})^\f{-1}{2^n-1}h_l\circ e^{\f{2\pi iz}{n}}-f\|_{B(0,R)} = \|\tilde c_n(h_l)^\f{-1}{2^n-1}h_l\circ e^{\f{2\pi iz}{n}}-h\circ e^{\f{2\pi iz}{n}}\|_{B(0,R)}\rightarrow 0.$$
Therefore, for large enough $l$, $c_n(h_l\circ e^{\f{2\pi iz}{n}})^\f{-1}{2^n-1}h_l\circ e^{\f{2\pi iz}{n}}\in U$. Finally by  Lemma \ref{periodic}, $h_l\circ e^{\f{2\pi iz}{n}}$ is $n$-periodic and by \eqref{periodic2} $c_n(h_l\circ e^{\f{2\pi iz}{n}})^\f{-1}{2^n-1}h_l\circ e^{\f{2\pi iz}{n}}$
is a periodic vector for $P$.

\subsection*{Proof that $P$ is frequently hypercyclic}
To prove the existence of frequently hypercyclic vectors we will use the following result \cite[Lemma 2.5]{BonGro07}.

\begin{lemma}
 There exist pairwise disjoint subsets $A_{n,m}$ of $\zN$, each having positive lower density such that for any $k\in A_{n,m}$, $k'\in A_{n',m'}$ we have
 $k>m$, and $|k-k'|>m+m'$ if $k\neq k'$.
\end{lemma}

We will now prove that $P$ supports a frequently hypercyclic vector. Our proof follows Example 9.6 in \cite{GroPer11} together with a careful  use of Runge's theorem.

 Let $A_{n,m}$ be the subsets given by the above lemma and consider $(k_j)_j\sub \zN$ the increasing sequence formed by
 $\bigcup A_{n,m}$. If $k_j\in A_{n,m}$ we define $B_j=B(k_j,r_j)$, where $r_j=\f{m}{2}+\f{1}{m}$ is a non natural radius. It follows from the above lemma that the $B_j$ are pairwise disjoint.
 Let $(p_n)_n$ be a dense sequence in $\H$ such that $p_n(l)\ne 0$ for every $l\in\mathbb Z$, $n\in\mathbb N$.

Applying Runge's Theorem recursively we will find $(f_j)_j\sub\H$ such that $f_j$ approximates $\tilde p_j(z):=p_n(z-k_j)$ on ${B_j}$, where $n$ is the only natural number such that $k_j\in A_{n,m}$,  such that $c_{k_j}(f)$ is close to $1$, and such that $f_j(l)\ne 0$ for every $l\in\mathbb Z$.
To achieve this, let $(\epsilon_j)_j\in \ell_1$ be a sequence of positive numbers such that   $\epsilon_j< \f{1}{m}$ whenever $k_j\in A_{n,m}$. 
We will define inductively a sequence of entire functions $(f_j)_j\subset\H$ and a sequence of positive numbers $(\delta_j)_j$ satisfying
\begin{itemize}
 \item[(a)] $\|f_{j+1}-f_j\|_{B(0,k_{j}+\f{1}{k_{j}})}<\delta_{j}$,
 \item[(b)] $\|c_{k_{j+1}}(f_{j+1})f_{j+1}-\tilde p_{j+1}\|_{B_{j+1}}<\delta_{j}$,
 \item[(c)] $\delta_{j+1}< \min\{\epsilon_{j+1},\epsilon_{j+2}/2,\gamma_{j+1}\}$,
 \item[(d)] $\delta_{j+1}<\gamma_{l_0}-\sum_{l=l_0}^{j}\delta_l$, for $l_0=1,\dots,j$ and
  \item[(e)] $f_{j+1}$ has no zero in $\mathbb Z$,
\end{itemize}
where $\gamma_j$ is a positive number that depends on $f_j$ as follows. For any $g\in\H$ and $j\in\mathbb N$, let
  $\Phi_{j}:\C^{k_j}\rightarrow \C$ defined as
 $$\Phi_{j}\l x_0,\ldots,x_{k_j-1}\r:=x_0^{2^{k_j-1}}\cdot\ldots \cdot x_{k_j-1}.$$ Thus, if we set
 $$K_{g,j}=\sup_{|z|<k_{j+1}} |g(z)|,$$
 we have that $\Phi_j$ is uniformly continuous on the product of the closed discs $\Pi_{l=1}^{k_j}\overline {B(0,K_{g,j}+\|\epsilon\|_1)}\subset \C^{k_j}$ and
 $$c_{k_j}(g)=\Phi_j(x_{g,j}),$$ 
 where $x_{g,j}$ is the vector $\l g(0),\ldots, g(k_j-1)\r$. 
 Since  $\Phi_j$ is uniformly continuous, given the number $\f{\epsilon_j}{2(K_{g,j}+\|\epsilon \|_1)}>0$ there exists $\gamma_{g,j}>0$ such that for every $x,y\in\overline {B(0,K_{g,j}+\|\epsilon\|_1)}\times\dots\times\overline {B(0,K_{g,j}+\|\epsilon\|_1)}$ we have that 
 \begin{equation}\label{def gamma}
  \text{ if } \|x-y\|_\infty<\gamma_{g,j}  \text{ then } |\Phi_j(x)-\Phi_j(y)|<\f{\epsilon_j}{2(K_{g,j}+\|\epsilon \|_1)}.
 \end{equation}
 Once fixed the function $f_j$, $\gamma_j$ will be defined as $\gamma_j:=\gamma_{f_j,j}$.

We start setting $f_1(z)=\tilde p_1(z)$ (thus we have defined $\gamma_1:=\gamma_{f_1,1}$). We define $\delta_1>0$ such that
 $$
 \delta_1<\min\{\epsilon_1,\epsilon_2/2,\gamma_1\}.
 $$

Suppose now that $f_1,\ldots f_j\in\H$ and $\delta_1,\ldots,\delta_j\in\zR_{>0}$ have been constructed and satisfy (a)-(e). We will now define $f_{j+1}$ and $\delta_{j+1}$.

 Consider $B^1_{j+1}$ and $B^2_{j+1}$ disjoint open sets so that
 $k_{j}+1 \in B^1_{j+1}$, $\{k_{j}+2,\ldots,k_{j+1}-\lfloor r_{j+1} \rfloor-1\} \sub B^2_{j+1}$,  and such that $\{|z|<k_{j}+\frac1{k_j}\},B^1_{j+1},B^2_{j+1},B_{j+1}$ are all disjoint.
 See \textbf{Fig.} \textbf{\ref{fig3}}.
\begin{figure}
 \centering
 \begin{tikzpicture}[scale=1] 
 \draw (0,0) -- (6.5,0);
  \draw [dashed](6.5,0) -- (9.5,0);
  \draw (9.5,0) -- (11.5,0);
  \draw [dashed] (11.5,0) -- (13,0);
   \draw  (13,0) -- (15,0);

  \draw [fill](4.1,0) circle [radius=0.025];
  \foreach \x in {5,...,11}
  \draw [fill](\x,0) circle [radius=0.025];
%   \node [below,scale=0.7] at (\x,0) {$k_j+\f{1}{k_j}$};
\draw [fill](14,0) circle [radius=0.025];
  \node (3)[scale=0.7] at (4.1,0) {};
  \node [below,scale=0.7] at (4,0) {$k_j+\f{1}{k_j}$};
  \node [below,scale=0.7] at (5,0) {$k_j+1$};
  \node [below,scale=0.7] at (6,0) {$k_j+2$};
  \node [below,scale=0.7] at (11.1,0) {$k_{j+1}-\lfloor r_{j+1} \rfloor$};
  \node [scale=0.7] at (10,0.2) {$k_{j+1}-\lfloor r_{j+1} \rfloor-1$};
  \node [scale=0.7] at (14,-0.2) {$k_{j+1}$};
  
  \node [scale=0.9] at (2.5,1) {$B(0,k_j+\frac{1}{k_j})$};
  \node at (12,0.7) {$B_{j+1}$};
  \node at (8,0.7) {$B^2_{j+1}$};
  \node at (5,0.7) {$B^1_{j+1}$};

%   \node (1) at (0,-3.1){1};
  \draw  (3) +(0:0) arc (0:30:3.1cm);
  \draw  (3) +(0:0) arc (360:330:3.1cm);
  
  \draw (5,0) circle [radius=0.40];
  \node (4)  at (5.5,-0.11) {};
  \node (5) at (10.2,-0.11){};
  \node (6)  at (5.5,0.11) {};
   \node (7)  at (10.2,0.11) {};

\draw (4) [-] to  [out=90,in=90] (5);
\draw (6) [-] to  [out=-90,in=-90] (7);
    
 \node (8)  at (10.8,0) {};
%  \draw (8) +(0:0) arc (300:360:2.2cm);
%  \draw (8) +(0:0) arc (45:90:2.2cm);
%  \draw (8) +(0:0) arc (90:45:2.2cm);
 \draw (8) +(0:0) arc (180:210:3.2cm);
 \draw (8) +(0:0) arc (-180:-210:3.2cm);

 \end{tikzpicture}
\caption{ The open sets $B(0,k_j+\frac{1}{k_j})$, $B^1_{j+1}$,  $B^2_{j+1}$ and $B_{j+1}$.}
\label{fig3}
\end{figure}
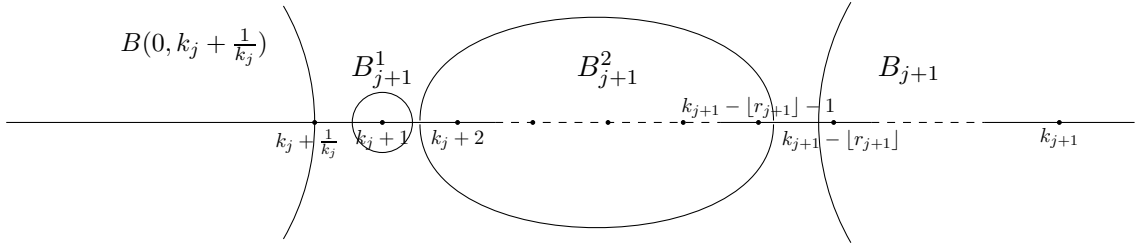

 Now by Runge's Theorem we can find auxiliary entire functions $g_l$ satisfying
 \begin{itemize}
  \item[(i)] $\|g_l-f_j\|_{B(0,k_{j}+\f{1}{k_{j}})}<\f{\epsilon_{j}}{l}$,
  \item[(ii)] $\sup_{z\in B^1_{j+1}} |g_l(z)-d_j|<\f{\epsilon_j}{l}$,
  \item[(iii)] $\sup_{z\in B^2_{j+1}} |g_l(z)-1|<\f{\epsilon_j}{l}$,
  \item[(iv)]  $\|g_l-\tilde p_{j+1}\|_{B_{j+1}} <\f{\epsilon_j}{l}$ and
  \item[(v)] $g_l$ has no zero in $\mathbb Z$;
 \end{itemize}
 where $1/d_j$ is a ($2^{k_{j+1}-k_j-2}$)-th root of the number 
 $$f_j(0)^{2^{k_{j+1}-1}}\cdot\ldots \cdot f_j(k_j)^{2^{k_{j+1}-k_j-1}}\cdot 1^{2^{k_{j+1}-k_j-3}}\cdot\ldots \cdot1^{2^{\lfloor r_{j+1}\rfloor}} \cdot \tilde p_{j+1}(k_{j+1}-\lfloor r_{j+1}\rfloor)^{2^{\lfloor r_{j+1}\rfloor-1}}\cdot\ldots\cdot  \tilde p_{j+1}(k_{j+1}-1),$$
 so that $c_{k_{j+1}}(g_l)$ approaches to $1$ as $l\to\infty$. Take now $l$ large enough so that $\frac{\epsilon_{j+1}}{l}<\delta_j$ and such that
\begin{align*}
\|c_{k_{j+1}}(g_l)g_l-\tilde p_{j+1}\|_{B_{j+1}} &\le |c_{k_{j+1}}(g_l)-1|\cdot\|g_l\|_{B_{j+1}}+\|g_l-\tilde p_{j+1}\|_{B_{j+1}} \\
 &  \le|c_{k_{j+1}}(g_l)-1|\cdot\l\|\tilde p_{j+1}\|_{B_{j+1}}+\f{\epsilon_j}{l}\r+\f{\epsilon_j}{l}<\delta_j.
\end{align*}
For such an $l$, we set $f_{j+1}:=g_{l}$, hence determining as above the number $\gamma_{j+1}>0$. Finally we set $\delta_{j+1}>0$ such that
$$
\delta_{j+1}<\min\left\{\epsilon_{j+1},\frac{\epsilon_{j+2}}{2},\gamma_{j+1},\gamma_{j}-\delta_j,\gamma_{j-1}-\delta_j-\delta_{j-1},\dots,\gamma_{1}-\sum_{l=1}^j\delta_l\right\}.
$$
This concludes the construction of $(f_j)_j$ and $(\delta_j)_j$ satisfying (a)-(e).
 
%  \textcolor{red}{los $l_j$ grandes para que $c_{k_j}(f)\sim 1$:
  Now we define $f$ as
 $$f:=f_1+\sum_{j=1}^\infty (f_{j+1}-f_j).$$
 Note that  (d) implies that 
$$\sum_{n\geq j} \delta_n\leq \gamma_j,$$ 
and in particular, the sequence $(\delta_j)_j\in \ell_1$. Thus, (a) and the fact that $k_j+\f{1}{k_{j}}\rightarrow \infty$ imply that $f$ is an entire function and that $f=\lim_{j\rightarrow \infty} f_j$. 
Moreover,
 $$c_{k_j}(f)=\Phi(x_{f,j}),$$
 where $x_{f,j}=\l f_j(0)+\sum_{n\geq j} (f_{n+1}(0)-f_n(0)),\ldots ,f_j(k_{j}-1)+\sum_{n\geq j} (f_{n+1}(k_{j}-1)-f_n(k_{j}-1))\r.$
Also, by (a) and (c), both $x_{f,j}$ 	and $x_{f_j,j}$ belong to $\overline {B(0,K_{f_j,j}+\|\epsilon\|_1)}\times\dots\times\overline {B(0,K_{f_j,j}+\|\epsilon\|_1)}$ and since
\begin{equation*}\label{cuenta final 2}
 \|x_{f,j}-x_{f_j,j}\|_{\infty}\leq \sup_{|z|<k_{j}-1} \left| \sum_{n\geq j} f_{n+1}(z)-f_n(z)\right|\leq  \sum_{n\geq j} \delta_{n} \le \gamma_j,	
\end{equation*}
 by \eqref{def gamma} we obtain 
\begin{equation}\label{cuenta final 3}
 | c_{k_{j}}(f)-c_{k_{j}}(f_j)|=  |\Phi_j (x_{f,j})-\Phi_j(x_{f_j,j})|\leq \f{\epsilon_j}{2(K_j+\|\epsilon \|_1)}.
\end{equation}
 Let $z\in B_j$, then using (a),(b),(c),\eqref{cuenta final 3},
\begin{align*}
 |c_{k_{j}}(f)f(z)- p_{n}(z-k_{j})|&\leq |c_{k_{j}}(f_j)f(z)- p_{n}(z-k_{j})|+|\l c_{k_{j}}(f)-c_{k_{j}}(f_j)\r f(z)|\\
 &\leq \delta_{j-1}+  \f{\epsilon_j}{2(K_j+\|\epsilon \|_1)}|f(z)|\\
 &\leq \f{\epsilon_j}{2}+\f{\epsilon_j}{2(K_j+\|\epsilon \|_1)}(|f_j(z)|+\sum_{n\ge j}|f_{n+1}(z)-f_n(z)|)\\
 &\leq  \f{\epsilon_j}{2}+\f{\epsilon_j}{2(K_j+\|\epsilon \|_1)}(K_j+\sum_{n\ge j}\epsilon_{n})\leq  {\epsilon_j}.
\end{align*}

 Therefore, for $k_j\in A_{n,m}$,
 $$\sup_{|z|<\frac{m}{2}+\frac1{m}} |P^{k_j}f(z)- p_{n}(z)|=\sup_{z\in B_j-k_j} |P^{k_j}f(z)- p_{n}(z)|= \sup_{z\in B_j} |c_{k_{j}}(f)f(z)- p_{n}(z-k_{j})| \le\epsilon_j<\f{1}{m}.$$
 Note that the sets
 $$U_{n,m}:=\left\{h\in \H: \sup_{|z|<\frac{m}{2}+\frac1{m}} |h(z)-p_n(z)|<\f{1}{m}\right\},$$
 with $n,m\in\mathbb N$, form a basis of open sets of $\H$. Finally since for $k\in A_{n,m}$, $P^k(f)\in U_{n,m}$   and each $A_{n,m}$ has positive lower density, we conclude that $f$ is a frequently hypercyclic vector for $P$.

 \section{Examples of non-hypercyclic polynomials on $\H$}
The purpose of this section is to show that many natural homogeneous polynomials on $\H$ fail to be hypercyclic. In view of what we have proved in the previous section, and the fact that translation and differentiation operators on $\H$ share many dynamical properties, a natural candidate to be hypercyclic is the homogeneous polynomial $P(f):= f(0)\cdot f'$. Another favorable motivation comes from the study of bilinear hypercyclic operators on $\H$. B{\`e}s and Conejero  considered in \cite[Section 4]{BesCon14} the bilinear operator $M(f,g)=f(0)g'$, and showed that it is hypercyclic (in the sense defined by the authors). Since $M(f,f)=P(f)$ it is reasonable to expect that $P$ is also hypercyclic.
Surprisingly, the polynomial fails to be hypercyclic. 
\begin{proposition}\label{no hiper}
 The homogeneous polynomial $P(f):=f(0)f'(z)$ is not hypercyclic.
\end{proposition}
\begin{proof}
	The iterates of a function $f$ are of the form
	$P^n(f)=c_n(f)f^{(n)}$, with
	$$c_n(f)=f(0)^{2^{n-1}}f'(0)^{2^{n-2}}\ldots f^{(n-1)}(0).$$ 
	This fact can be easily proven by induction. Also the functions $c_n(f)$ can be
	constructed recursively as
	$$\begin{cases}
	c_1(f)=f(0);\\
	c_n(f)=c_{n-1}(f)^2 f^{(n-1)}(0).
	\end{cases}$$
	Let us define $X\sub \H$ as $X:=\{f\in \H: \limsup |c_n(f)(n!)^2R^n|=\infty, \textrm{ for some }R>1\}$.
	The proof of the proposition will be divided in three steps:
	\begin{enumerate}
		\item\label{s1} $X$ is $P$-invariant.
		\item\label{s2} $0$ is not in the closure of $X$.
		\item\label{s3} If $f\notin X$ then $f$ is not a hypercyclic vector for $P$.
	\end{enumerate}
	If we prove \eqref{s1}, \eqref{s2} and \eqref{s3} it clearly follows that $P$ is not hypercyclic.

	\emph{Proof of  \eqref{s1}.} Note that $c_{k+1}(f)=f(0)c_k(P(f))$. Take $f\in X$ and let $R>1$ such that $\limsup |c_n(f)(n!)^2R^n|=\infty$. Then $Pf\in X$, because
	$$
	\limsup \left|c_n(Pf)(n!)^2(R+1)^n\right|= \limsup \left|c_{n+1}(f)((n+1)!)^2R^{n+1}\frac{(R+1)^n}{f(0)(n+1)^2R^{n+1}}\right|=\infty.
	$$

	\emph{Proof of \eqref{s2}.}
	Suppose that $(f_k)_k \subseteq X$ is a null sequence. Since $c_n$ is continuous, there exists $(k_n)_n\sub \zN$ with $|c_n(f_{k_n})|<\frac{1}{2^{2^n}}.$
	Taking a subsequence, we may suppose that
	$$|c_n(f_{n})|<\frac{1}{2^{2^n}}.$$
	
	We claim that for each $n$, there exists $j\ge 0$ such that $|f_n^{(j+n)}(0)|>(j+n)^{j+n}$. Indeed, if $|f_n^{(j+n)}(0)|\le (j+n)^{j+n}$ for every $j\ge 0$ then we show by induction that
	$|c_{j+n}(f_n)|<\frac{1}{2^{2^{n+\frac{j}{2}}}}$ for every $j\ge 0$. We already know it for $j=0$. Suppose it is true for some $j$. Note that for every $n,j$, we have
	$$\frac{2^{(n+j) \log_2(n+j)}}{2^{(\sqrt 2-1)2^{n+\frac{1+j}{2}}}}\leq 1.$$
	Thus
	\begin{align*}
	|c_{n+j+1}(f_n)|&=|c_{n+j}^2(f_n)| |f_n^{(j+n)}(0)|
	\leq \frac{1}{2^{2^{n+1+\frac{j}{2}}}} |f_n^{(j+n)}(0)|\\
	&\leq\frac{1}{2^{2^{n+1+\frac{j}{2}}}} 2^{(n+j) \log_2(n+j)}\\
	&=\frac{2^{(n+j) \log_2(n+j)}} {2^{2^{n+\frac{j}{2}+\frac{1}{2}}} 2^{(\sqrt 2-1)2^{n+\frac{j}{2}+\frac{1}{2}}}}
	\leq \frac{1}{2^{2^{n+\frac{j+1}{2}}}}.
	\end{align*}
	This implies that $f_n$ is not in $X$, which is a contradiction.
	
	Therefore   $|f_n^{(j+n)}(0)|>(j+n)^{j+n}$ for some $j\ge 0$.
	Recall that the seminorms given by 
	$$\|f\|_k=\sup_{j} |f^{(j)}(0)| \frac{k^j}{j!},$$
	define the topology of $\H$ (see for example \cite[Example 27.27]{MeiVog97} or \cite{Per99}).
	
	For each $n$, let $j_n\ge 0$ be such that  $|f_n^{(j_n+n)}(0)|>(j_n+n)^{j_n+n}$. Thus,
	\begin{equation*}
	\|f_{n}\|_k=\sup_j |f_n^{(j)}(0)| \frac{k^j}{j!}>|f_n^{(j_n+n)}(0)| \frac{k^{j_n+n}}{(n+j_n)!}>1.
	\end{equation*}
	This contradicts the fact that $f_n\to0$.

	\emph{Proof of  \eqref{s3}.}
	Suppose that $|c_n(f)(n!)^2|\le L<\infty$ for every $n\ge 0$. Then, by the Cauchy inequalities, we have for some $M,r>0$,
	$$
	|\delta_0(P^{n+1}f)| =|c_{n+1}(f)f^{(n+1)}(0)| \le \frac{L}{(n!)^2}\frac{M(n+1)!}{r^{n+1}}\to 0.
	$$
	Therefore $f$ is not a hypercyclic vector.
\end{proof}

Aron and Miralles \cite{AroMir08} showed that the polynomial $P\in \PP(^2C^k(\zR))$ defined as $P(f)(z)=f(z+1)^2$ is hypercyclic. However, if we consider the analogous map, but in $\H$, the polynomial
 fails resoundingly to be hypercyclic. The rigidity of the holomorphic functions obstructs our search of hypercyclic homogeneous polynomials. In particular, Hurwitz's Theorem impose several restrictions to this kind of problem. 
This was already noted  in \cite{AroConPer07}, as the authors   were looking for algebras of hypercyclic vectors.
 \begin{proposition}
	Let $a,b\in \mathbb C$ and let $P\in \PP(^2\H)$ be  the polynomial defined by 
	$$
	P(g)(z)=g(z+a)g(z+b).
	$$
If $f$ is an accumulation point of an orbit of $P$ then  either $f$ is identically zero or $f(z)\ne0$ for every  $z\in\mathbb C$. In particular, $P$  is not hypercyclic.
\end{proposition}
\begin{proof}
	Note that if $g\in \H$, then 
	 \begin{equation}\label{hurw}
	 \displaystyle P^k(g)(z)=\prod_{j=0}^k  g(z+ja+(k-j)b)^{\binom{k}{j}}.
	 \end{equation}
	 Let $f\in \H$ and suppose that $f$ has a zero of order $m\ge 1$ at $z_0$. If  $P^{k_l}(g)$ converges uniformly to $f$ on $B(z_0,2(|a|+|b|))$, by Hurwitz's Theorem, for each sufficiently small $\delta>0$, there is some $l_0$ such that for $l\geq l_0$ the number of zeros of $P^{k_l}(g)$ in $B(z_0,\delta)$ is exactly  $m$. Thus for each $l\geq l_0$ there is some $j_l\le k_l$ such that  $g(\cdot+j_la+(k_l-j_l)b)$ has a zero of positive order in $B(z_0,\delta)$. But this implies, by \eqref{hurw}, that   $P^{k_l}(g)$ must have another zero of order $\ge k_l$ in  $B(z_0+a-b,\delta)$ (or in $B(z_0-a+b,\delta)$ if $j_l=k_l$), and therefore  $f$ must be identically zero.
	 \end{proof}

It was proved in \cite{AroConPer07} that, in contrast with the translation operator, the differentiation operator does admit an algebra of hypercyclic vectors. However, Hurwitz's Theorem also prevents powers of the differentiation operator to be hypercyclic.
 \begin{proposition}
 Let $P\in \PP(^2\H)$ be one of the following polynomials
 \begin{itemize}
  \item[(i)]   $P(g)(z)=g'(z)^2$,
  \item[(ii)]  $P(g)(z)=g(z)g'(z).$ 
 \end{itemize}
Then, $P$  is not hypercyclic. 
\end{proposition}
\begin{proof}
We only prove (i), the proof of (ii) is analogous.  Note that if $g'(z_0)=0$ then $(P(g))'(z_0)=0$. Suppose that $P^{n_k}(g)\to z^2$. Then $P^{n_k}(g)'\to 2z$. By Hurwitz's Theorem, there exists $k_0$ such that for every $k\ge k_0$, $P^{n_k}(g)'$ has a zero of order 1 in $B(0,1)$. Thus $P^{n}(g)'$ has a zero of order at least 1 in $B(0,1)$ for every $n\ge n_{k_0}$. Therefore $g$ is not hypercyclic for $P$.
\end{proof}


\begin{thebibliography}{10}
	
	\bibitem{AroConPer07}
	R.~M. Aron, J.~A. Conejero, A.~Peris, and J.~B. Seoane-Sep{\'u}lveda.
	\newblock Powers of hypercyclic functions for some classical hypercyclic
	operators.
	\newblock {\em Integral Equations and Operator Theory}, 58(4):591--596, 2007.
	
	\bibitem{AroMar04}
	R.~M. Aron and D.~Markose.
	\newblock On universal functions.
	\newblock {\em J. Korean Math. Soc.}, 41(1):65--76, 2004.
	\newblock Satellite Conference on Infinite Dimensional Function Theory.
	
	\bibitem{AroMir08}
	R.~M. {Aron} and A.~{Miralles}.
	\newblock {Chaotic polynomials in spaces of continuous and differentiable
		functions.}
	\newblock {\em {Glasg. Math. J.}}, 50(2):319--323, 2008.
	
	\bibitem{BayGri04}
	F.~Bayart and S.~Grivaux.
	\newblock {Hypercyclicity: the role of the unimodular point spectrum.
		(Hypercyclicit\'e : Le r\^ole du spectre ponctuel unimodulaire.)}.
	\newblock {\em C. R., Math., Acad. Sci. Paris}, 338(9):703--708, 2004.
	
	\bibitem{BayGri06}
	F.~Bayart and S.~Grivaux.
	\newblock Frequently hypercyclic operators.
	\newblock {\em Trans. Amer. Math. Soc.}, 358(11):5083--5117 (electronic), 2006.
	
	\bibitem{BayMat09}
	F.~Bayart and E.~Matheron.
	\newblock {\em {Dynamics of linear operators.}}
	\newblock {Cambridge Tracts in Mathematics 179. Cambridge: Cambridge University
		Press. xiv, 337~p.}, 2009.
	
	\bibitem{bernal2005backward}
	L.~Bernal-Gonz{\'a}lez.
	\newblock Backward $\phi$-shifts and universality.
	\newblock {\em Journal of mathematical analysis and applications},
	306(1):180--196, 2005.
	
	\bibitem{Ber98}
	N.~C. {Bernardes}.
	\newblock {On orbits of polynomial maps in Banach spaces.}
	\newblock {\em {Quaest. Math.}}, 21(3-4):311--318, 1998.
	
	\bibitem{BerPer13}
	N.~C. Bernardes and A.~Peris.
	\newblock On the existence of polynomials with chaotic behaviour.
	\newblock {\em Journal of Function Spaces and Applications}, 2013, 2013.
	
	\bibitem{BesCon14}
	J.~B{\`e}s and J.~A. Conejero.
	\newblock An extension of hypercyclicity for {$N$}-linear operators.
	\newblock In {\em Abstract and Applied Analysis}, volume 2014. Hindawi
	Publishing Corporation, 2014.
	
	\bibitem{Bir29}
	G.~D. Birkhoff.
	\newblock {D\'emonstration d'un th\'eor\`eme \'el\'ementaire sur les fonctions
		enti\`eres.}
	\newblock {\em C. R Acad. Sci. Paris}, 189:473--475, 1929.
	
	\bibitem{BonGro06}
	A.~Bonilla and K.-G. Grosse-Erdmann.
	\newblock On a theorem of {G}odefroy and {S}hapiro.
	\newblock {\em Integral Equations Operator Theory}, 56(2):151--162, 2006.
	
	\bibitem{BonGro07}
	A.~{Bonilla} and K.-G. {Grosse-Erdmann}.
	\newblock {Frequently hypercyclic operators and vectors.}
	\newblock {\em {Ergodic Theory Dyn. Syst.}}, 27(2):383--404, 2007.
	
	\bibitem{GodSha91}
	G.~Godefroy and J.~H. Shapiro.
	\newblock Operators with dense, invariant, cyclic vector manifolds.
	\newblock {\em J. Funct. Anal.}, 98(2):229--269, 1991.
	
	\bibitem{GroKim13}
	K.-G. Grosse-Erdmann and S.~G. Kim.
	\newblock Bihypercyclic bilinear mappings.
	\newblock {\em Journal of Mathematical Analysis and Applications},
	399(2):701--708, 2013.
	
	\bibitem{GroPer11}
	K.-G. Grosse-Erdmann and A.~Peris~Manguillot.
	\newblock {\em {Linear chaos.}}
	\newblock {Universitext. Berlin: Springer. xii, 386~p. EUR~53.45 }, 2011.
	
	\bibitem{jung2017mixing}
	A.~Jung.
	\newblock Mixing, simultaneous universal and disjoint universal backward
	$\phi$-shifts.
	\newblock {\em Journal of Mathematical Analysis and Applications},
	452(1):246--257, 2017.
	
	\bibitem{kim2012numerically}
	S.~G. Kim, A.~Peris, and H.~G. Song.
	\newblock Numerically hypercyclic polynomials.
	\newblock {\em Archiv der Mathematik}, pages 1--10, 2012.
	
	\bibitem{Mac52}
	G.~R. MacLane.
	\newblock Sequences of derivatives and normal families.
	\newblock {\em J. Analyse Math.}, 2:72--87, 1952.
	
	\bibitem{MarPer09}
	F.~Mart{\'\i}nez-Gim{\'e}nez and A.~Peris.
	\newblock Existence of hypercyclic polynomials on complex fr{\'e}chet spaces.
	\newblock {\em Topology and its Applications}, 156(18):3007--3010, 2009.
	
	\bibitem{MarPer10}
	F.~Mart{\'\i}nez-Gim{\'e}nez and A.~Peris.
	\newblock Chaotic polynomials on sequence and function spaces.
	\newblock {\em International Journal of Bifurcation and Chaos},
	20(09):2861--2867, 2010.
	
	\bibitem{MeiVog97}
	R.~Meise and D.~Vogt.
	\newblock {\em Introduction to functional analysis}.
	\newblock Clarendon Press, 1997.
	
	\bibitem{Per99}
	A.~{Peris}.
	\newblock {Chaotic polynomials on Fr\'echet spaces.}
	\newblock {\em {Proc. Am. Math. Soc.}}, 127(12):3601--3603, 1999.
	
	\bibitem{Per01}
	A.~{Peris}.
	\newblock {Erratum to: ``Chaotic polynomials on Fr\'echet spaces''.}
	\newblock {\em {Proc. Am. Math. Soc.}}, 129(12):3759--3760, 2001.
	
	\bibitem{peris2003chaotic}
	A.~Peris.
	\newblock Chaotic polynomials on banach spaces.
	\newblock {\em Journal of mathematical analysis and applications},
	287(2):487--493, 2003.
	
\end{thebibliography}
 \end{document}